\documentclass[fleqn,a4paper,12pt]{article}
\usepackage{amsmath,amsfonts,calrsfs,amssymb,color}
\usepackage{amsthm}

\newtheorem{Theorem}{Theorem} 
\newtheorem{Definition}[Theorem]{Definition}
\newtheorem{Proposition}[Theorem]{Proposition}
\newtheorem{Lemma}[Theorem]{Lemma}

\newtheorem{Remark}[Theorem]{Remark}

\newtheorem{Hypothesis}[Theorem]{Hypothesis}
\makeatletter

\def\R{\mathbb R}
\def\N{\mathbb N}

\def\E{\mathbb E}
\def\P{\mathbb P}
\def\ds{\displaystyle}

\title{Existence of the Fomin derivative of the invariant measure  of a stochastic reaction--diffusion equation}

\author{Giuseppe Da Prato \thanks{Giuseppe Da Prato, Scuola Normale Superiore, 56126, Pisa, 
Italy.   e-mail:   giuseppe. daprato@sns.it }  and Arnaud Debussche \thanks{
Arnaud Debussche, IRMAR and \'Ecole Normale Sup\'erieure de Rennes, Campus de Ker Lann, 37170 Bruz, France. e-mail:arnaud.debussche@ens-rennes.fr}}

\begin{document}

\maketitle

\begin{abstract}
We consider a reaction--diffusion equation  perturbed by noise (not necessarily white). We prove existence of the Fomin derivative of the corresponding transition semigroup $P_t$.  The main tool is a   new  estimate for $P_tD\varphi$ in terms of  $\|\varphi\|_{L^2(H,\nu)}$, where $\nu$ is the invariant measure of $P_t$.

\end{abstract}

\noindent {\bf 2010 Mathematics Subject Classification AMS}: 60H15, 35R15, 35K57\medskip


 \section{Introduction}

Let $H=L^2(\mathcal O)$ where $\mathcal O=[0,1]^n$, $n\in \N$  \footnote{This choice is made for simplicity, all result below hold for a bounded domain of $\mathcal O$ with sufficiently regular boundary (Lipschitz for instance).}, and denote by  $\partial \mathcal O$ the
boundary of $\mathcal O$.
  We are concerned with the following stochastic differential  equation
\begin{equation}
\label{e4.1}
\left\{\begin{array}{lll}
dX(t)=[AX(t)+p(X(t))]dt+BdW(t),\\
\\
X(0)=x.
\end{array}\right.
\end{equation}
where  $A$ is  the realization of the Laplace operator  $\Delta_\xi $ equipped with
Dirichlet boundary conditions,
$$
Ax=\Delta_\xi x,\quad x\in D(A),\quad D(A)=H^2(\mathcal O)\cap H^1_0(\mathcal O),
$$
  $p$ is a decreasing polynomial   
of odd degree equal to $N>1$,  $B\in L(H)$ and  $W$  is an $H$--valued cylindrical Wiener process  on a filtered
probability space
$(\Omega,\mathcal F, (\mathcal F_t)_{t>0},\P)$. 
 
   It is well known that this equation has unique strong solutions and that the associate transition semigroup possesses a unique invariant measure.

 The aim of this article is to derive new properties on this invariant measure. If $B$ is the identity, 
 then the system is gradient and the invariant measure is explicit but this is not the case in general. If $B$ commutes with $A$ and has a bounded inverse, it follows from \cite{BogDaPRock} that the invariant 
 measure has a density which is in a Sobolev space based on the reference gaussian measure associated to the linear equation. It has also been shown in \cite{DaDe-04} that under our assumptions, the invariant measure is absolutely continuous with respect to the reference gaussian measure. Otherwise, not much is
 known on this invariant measure.

 For as   the  operator $B$ is concerned, we shall assume: 
\begin{Hypothesis}
\label{h1}
  $B=(-A)^{-\gamma/2}$  where   $\frac{n}2-1<\gamma<1$.  
Obviously this implies that $n<4$.
 
 \end{Hypothesis}

 \begin{Remark}
 \em The assumption $\frac{n}2-1<\gamma$ implies that the stochastic convolution
$$
W_A(t):=\int_0^t(-A)^{-\gamma/2} e^{(t-s)A}dW(s),\quad t\ge 0,
$$
 is a well defined continuous process see e.g. \cite{DaZa14},  whereas under the condition $\gamma<1$    the Bismut--Elworthy-Li formula \eqref{e4.10} below holds and implies strong Feller 
 property on $H$, see \cite {Ce01}.
 If $\gamma\ge 1$, we need to work with different topologies.
 
 If $\gamma <\frac{n}2-1$, equation \eqref{e4.1} is not expected to have solutions with positive 
 spatial regularity and the equation has to be renormalized. This has been studied in \cite{DaDe-03} for $n=2$ and more recently in \cite{Hai14} and \cite{CaCh14} for $n=3$. 
 
 All following results remain true taking $B=G(-A)^{-\gamma/2}$ with $G\in L(H)$ and $\frac{n}2-1<\gamma<1$. We take this form for $B$ for simplicity.
 
 Also, the assumption that $p$ is decreasing is not necessary and could be replaced by: $p'$ is bounded above.
 \end{Remark}

Before explaining the content of the paper, it is convenient to recall some   results about    problem \eqref{e4.1}, that we gather from \cite{Da04}. We notice, however, that Reaction--Diffusion equations have been recently the object of several researches,  see \cite{DaZa14} and references therein.

We start with the definition of   solution of \eqref{e4.1}.
\begin{Definition}
\label{d4.5}
(i). Let $x\in L^{2N}(\mathcal O);$ we say that $X\in C_W([0,T];H)$
$\footnote{By $C_W([0,T];H)$ we  mean the set of  $H$--valued stochastic processes continuous in mean square and adapted to the filtration $(\mathcal F_t)$ .}$ is a
 {\em mild} solution
of problem \eqref{e4.1} if $X(t)\in L^{2N}(\mathcal O)$ for all $t\ge 0$ and  fulfills the following integral equation
\begin{equation}
\label{e4.8}
X(t)=e^{tA}x+\int_0^te^{(t-s)A}p(X(s))ds+W_A(t),\quad t\ge 0.
\end{equation}

\noindent (ii). Let $x\in H;$ we say that $X\in C_W([0,T];H)$ is a {\em generalized} solution
of problem \eqref{e4.1} if there exists a sequence $(x_n)\subset L^{2N}(\mathcal O),$ such that
$$
\lim_{n\to \infty}x_n=x\quad\mbox{\rm in}\;L^{2}(\mathcal O),
$$
and
$$
\lim_{n\to \infty}X(\cdot, x_n)=X(\cdot,x)\quad\mbox{\rm in}\;C_W([0,T];H).
$$

\end{Definition}

It is convenient to introduce  the following approximating problem
\begin{equation}
\label{e4.9}
\left\{\begin{array}{l}
dX_\alpha (t)=(AX_\alpha (t)+p_\alpha (X_\alpha (t))dt+(-A)^{-\gamma/2}dW(t),\\
\\
X_\alpha (0)=x\in H,
\end{array}\right.
\end{equation}
where for any $\alpha >0,$   $p_\alpha$  are the Yosida approximations of $p$, that is
$$
p_{\alpha }(r)= \frac{1}{\alpha }
\;(r-J_{\alpha}(r)),\;J_{\alpha }(r)=(1-\alpha p(\cdot))^{-1}(r),\quad r\in \R.
$$
Notice that, since $p_{\alpha }$ is Lipschitz continuous, then 
 for any $\alpha >0,$ and any $x\in H,$
  problem $(\ref{e4.9})$ has a unique 
 solution $X_\alpha (\cdot,x)\in C_W([0,T];H)$.

The following result is proved in \cite[Theorem 4.8]{Da04}
\begin{Proposition}
\label{p4.8} 
Assume that Hypothesis $\ref{h1}$ holds and let $T>0$. Then 
\begin{enumerate}

\item[(i)] If $x\in L^{2N}(\mathcal O),$  problem
  $(\ref{e4.1})$ has a unique  mild solution $X(\cdot,x)$.    
  
  \item[(ii)] If $x\in L^{2}(\mathcal O),$  problem
  $(\ref{e4.1})$ has a unique  generalized solution $X(\cdot,x).$

  In both cases $\ds{\lim_{\alpha \to 0}X_\alpha (\cdot,x)=X(\cdot,x)}$ in 
  $C_W([0,T];H).$

\end{enumerate}

\end{Proposition}

We introduce now the transition semigroup $P_t$
$$
P_t\varphi(x)=\E[\varphi(X(t,x))],\quad \varphi\in B_b(H)
$$ 
and the approximate transition semigroup $P^\alpha_t$
$$
P_t^\alpha\varphi(x)=\E[\varphi(X_\alpha(t,x))],\quad \varphi\in B_b(H).
$$
By $B_b(H)$ we mean the space of all $H$--valued real mappings that are Borel and bounded.

 For  $P^\alpha_t$
 the following  Bismut-Elworthy-Li formula holds, see \cite{Ce01}.
 \begin{equation}
\label{e4.10}
 \langle DP^\alpha _t\varphi(x),h  \rangle=\frac1t\;\E\left[\varphi(X_\alpha (t,x))\int_0^t\langle 
(-A)^{\frac{\gamma}{2}}\eta_\alpha^h(s,x),dW(s) \rangle  \right],\quad h\in H, 
\end{equation}
 where for any $h\in H$,  $\eta_\alpha^h(t,x)=:D_xX_\alpha(t,x)\cdot h$  is the differential of $X_\alpha(t,x)$ with respect to $x$ in the direction $h$ and is the  solution of the equation
 \begin{equation}
\label{e5b}
 \frac{d}{dt}\;\eta_\alpha ^{h}(t,x)=A\eta_\alpha ^{h}(t,x)-p'_\alpha(X_\alpha(t,x))
 \eta_\alpha^{h}(t,x),\quad \eta_\alpha^{h}(0,x)=h.
\end{equation}

 The following result is proved in \cite[Theorem
 4.16]{Da04}
 \begin{Proposition}
\label{p4.16} 
Assume that Hypothesis $\ref{h1}$ holds. Then  the semigroup $P_t$ has a unique invariant
measure $\nu .$ Moreover  there exists $c_N>0$
such that
 \begin{equation}
\label{e4.11}
\int_H|x|_{L^{2N}(\mathcal O)}^{2N}\nu(dx)\le c_N.
\end{equation}
\end{Proposition}
 \noindent Similarly   the approximating problem \eqref{e4.9} has a unique  invariant
measure $\nu_\alpha$.  It is not difficult to show  that $\nu_\alpha$ weakly converges to $\nu$  and
 \begin{equation}
\label{e7c}
\int_H|x|_{L^{2N}(\mathcal O)}^{2N}\nu_\alpha(dx)\le c_N.
\end{equation}
However, since we couldn't find a quotation of this fact, we have added a proof in the Appendix below.

   As well known  $P_{t}$ can be uniquely extended to a
strongly  continuous  semigroup of contractions in  $L^{2}(H,\nu)$ (still denoted $P_t$).
We shall denote by $\mathcal L$ its 
 infinitesimal generator   and by  $\mathcal L_0$   the differential operator    
$$
\begin{array}{lll}
\mathcal L_0\varphi&=&\ds{\frac 12\;\mbox{\rm Tr}\;[(-A)^{-\gamma}D ^2\varphi]+\langle x, AD \varphi\rangle
+\langle p(x), D \varphi\rangle}\ ,\quad
\varphi\in  \mathcal E_A(H), 
\end{array}
$$
 where  $ \mathcal E_A(H)$  is 
  the linear span of    all real parts of  functions of the form 
   $$\varphi_h(x):=e^{i \langle h,x   \rangle},\;x\in H,
$$
where  $h\in D(A)$.  We have used the notation $D\varphi$ for the gradient of $\varphi$ in $H$.

Similarly, for any $\alpha>0$,  $P^\alpha_{t}$ can be uniquely extended to a
strongly  continuous  semigroup of contractions in  $L^{2}(H,\nu)$ whose
 infinitesimal generator we denote by  ${\mathcal L}^\alpha$. We denote by ${\mathcal L}^\alpha _0$   the differential operator   defined by 
$$
\begin{array}{lll}
{\mathcal L}^\alpha _0&=&\ds{\frac 12\;\mbox{\rm Tr}\;[(-A)^{-\gamma}D ^2\varphi]+\langle x, AD \varphi\rangle
+\langle p_\alpha (x), D \varphi\rangle}\ ,\quad
\varphi\in  \mathcal E_A(H),\;x\in H.
\end{array}
$$

\begin{Proposition}
  \label{p4.23}
Assume   that   Hypothesis  $\ref{h1}$
  holds.
Then   $\mathcal L$ is the closure of $\mathcal L_0$ in $L^2(H,\nu)$  and  ${\mathcal L}^\alpha$ is the closure of ${\mathcal L}^\alpha_0$ in $L^2(H,\nu_\alpha).$
 \end{Proposition}
  The first assertion of the proposition is proved in \cite[Theorem 4.23]{Da04}, the proof of the latter is completely similar and so, it is omitted.\bigskip

 Now we are ready to describe the main goal of the paper. First, we shall prove the following integration by parts formula for the invariant  measure $\nu$.  
For any  $h\in H$ and  any  $\varphi\in C^1_b(H)$ there exists a function  $v^h\in L^2(H,\nu)$ such that
\begin{equation}
\label{e8c}
\int_H \langle (-A)^{-1} D\varphi(x),h\rangle\,\nu(dx)=\int_H \varphi(x)\,v^h(x)\,\nu(dx).
\end{equation}
Then we deduce by  \eqref{e8c}
   the existence  of  the Fomin derivative of $\nu$ in any direction $A^{-1}h$. \footnote{For the definition of  Fomin derivative see e.g.  \cite{Pu98}. }

A similar result, concerning the Burgers equation driven by white noise, has been proved in \cite{DaDe14}. In the present case  the proof of \eqref{e8c} is based, as in \cite{DaDe14}, on an estimate of $P_tD\varphi$ depending only on $\|\varphi\|_{L^2(H,\nu)}$. However, the techniques  used here   are obviously different.

We believe that our method could be used for other SPDEs as: singular dissipative equations and $2D$--Navier--Stokes equations. Both will be the object of future work.
  \medskip

In Section 2 we prove an identity relating
$ DP^\alpha_t\varphi$ and $P^\alpha_tD\varphi$. Using this identity in  Section 3 we prove the estimate
 \begin{equation}
\label{e14e}
\int _H \langle D\varphi(x),h\rangle  \, \nu (dx) \le C\|\varphi\|_{L^2(H,\nu)} \,|Ah|_H,\quad \varphi\in L^2(H,\nu).
\end{equation}
Finally, Section 4 is   devoted to show some consequences as the definition of Sobolev space with respect to the measure $\nu$.

\section{An identity relating $DP_t^\alpha\varphi$ and  $P_t^\alpha D\varphi$}

  \begin{Proposition}
  \label{p1}
  For any  $\varphi\in C^1_b(H)$,  $\alpha>0$, $h,x\in D(A)$,  we have
  \begin{equation}
\label{e7}
\begin{array}{l}
  \ds P^\alpha_t(\langle D\varphi(x),h\rangle)= \langle DP^\alpha_t\varphi(x), h\rangle 
  \ds-   \int_0^tP^\alpha_{t-s}(\langle Ah+Dp_\alpha(x)h  ,D P^\alpha_s\varphi(x)\rangle)ds,
  \end{array}
\end{equation}
where $p^\alpha$ are the Yosida approximations of $p$.
   
\end{Proposition}
\begin{proof}

      Let     $\varphi\in\mathcal E_A(H)$, $u^\alpha (t,x)=P^\alpha _t\varphi(x)$. Then $\varphi\in D({\mathcal L}^\alpha )$ and so,
  $$
  D_tu^\alpha (t,x)={\mathcal L}^\alpha  u^\alpha(t,x)=
  P^\alpha _t\left( \frac12\mbox{\rm Tr}\;[(-A)^\gamma D^2\varphi(x)]+\langle Ax+p^\alpha (x), D\varphi(x)\rangle   \right)
  $$
  Now let  $h\in D(A)$. Then setting  $v^\alpha _h(t,x)=\langle DP^\alpha _t\varphi(x),h\rangle$, we have
  $$
  D_tv^\alpha _h(t,x)= {\mathcal L}^\alpha  v_h^\alpha(t,x) +\langle Ah+p'_\alpha (x)h, Du^\alpha (t,x)\rangle 
  $$
  and by the variation of constants formula we deduce that 
  $$
  v^\alpha _h(t,x)=P^\alpha_tv^\alpha _h(0,x)+ \int_0^tP^\alpha _{t-s}(\langle Ah+p'_\alpha (x)h  ,D P^\alpha _s\varphi(x)\rangle)ds
  $$
  which is equivalent to
  $$
\begin{array}{l}
  \ds P^\alpha _t(\langle D\varphi(x),h\rangle)= \langle DP^\alpha _t\varphi(x), h\rangle\\
  \\
  \ds-   \int_0^tP^\alpha _{t-s}(\langle Ah+p'_\alpha (x)h  ,D P^\alpha _s\varphi(x)\rangle)ds.
  \end{array} 
$$  for all  $\varphi\in \mathcal E_A(H).$ 
  Since   $\mathcal E_A(H)$ is a core for  ${\mathcal L}^\alpha $ (Proposition \ref{p4.23}), the conclusion follows.
  
  \end{proof} 
  
  \begin{Remark}
  \em Probably identity  \eqref{e7} could be useful also in finite dimensions for SDEs with non degenerate noise. In fact  \eqref{e7} looks   simpler  than the formula obtained via   Maliavin Calculus, even if the latter allows to consider non degenerate equations, see \cite{Ma97}, \cite{Sa05}.
  
  \end{Remark} 
  
  \section{The main result}
  
  We first need a lemma.
  \begin{Lemma}
\label{l7b}
For any $\alpha>0$, $T>0$  and any $h\in H$ we have
\begin{equation}
\label{e11d}
|\eta_\alpha ^{h}(T,x)| \le  |h|,\quad x\in H,
\end{equation}
and
\begin{equation}
\label{e4.17}
\int_0^T |(-A)^{1/2}\eta_\alpha ^{h}(t,x)|^2dt\le  |h|^2,\quad x\in H.
\end{equation}
Finally, for any $\beta\in (0,1/2)$   we have
\begin{equation}
\label{e13d}
\int_0^T |(-A)^{\beta}\eta_\alpha ^{h}(t,x)|^2dt\le  C_{T,\beta} T^{1-2\beta}|h|^2,\quad x\in H.
\end{equation}
 
\end{Lemma}
\begin{proof}
By \eqref{e5b} and $p'_\alpha\le 0$, we have
 \begin{equation}
\label{e9b}
\ds{\frac12\frac{d}{dt}\;|\eta_\alpha ^{h}(t,x)}|^2+\frac12|(-A)^{1/2}\eta_\alpha ^{h}(t,x)|^2=\int_\mathcal O p'_\alpha(X_\alpha(t,x)) 
 (\eta_\alpha^{h}(t,x))^2d\xi\le 0.
\end{equation}
Integrating  in $t$ from $0$ to $T$, yields
$$
|\eta_\alpha ^{h}(T,x)|^2+\int_0^T|(-A)^{1/2}\eta_\alpha ^{h}(t,x)|^2\,dt\le |h|^2.
$$
So, \eqref{e11d} and \eqref{e4.17} follow. It remains to show \eqref{e13d}.

Let  us recall a well known  estimate from interpolation.
For $0<\beta<1/2$ we have
\begin{equation}
\label{e15d}
|(-A)^\beta x|\le |x|^{1-2\beta}\;|(-A)^{1/2} x|^{2\beta},\quad\forall\;x\in D((-A)^{1/2}).
\end{equation}
It follows that
$$
 \int_0^T |(-A)^{\beta}\eta_\alpha ^{h}(t,x)|^2dt\le \int_0^T |\eta_\alpha ^{h}(t,x)|^{1-2\beta}\;|(-A)^{1/2} \eta_\alpha ^{h}(t,x)|^{2\beta}dt.
$$
Recalling \eqref{e11d}  and using H\"older's inequality it follows that
$$
\begin{array}{l}
 \ds\int_0^T |(-A)^{\beta}\eta_\alpha ^{h}(t,x)|^2dt\\
 \\
 \ds\le \int_0^T |(-A)^{1/2} \eta_\alpha ^{h}(t,x)|^{2\beta}dt\;|h|^{1-2\beta}\\
 \\
 \ds\le T^{1/2-\beta}\left[ \int_0^T |(-A)^{1/2} \eta_\alpha ^{h}(t,x)|^{2}dt  \right]^\beta\;|h|^{1-2\beta} 
 \end{array}
$$
Finally, taking into account \eqref{e4.17}, yields
\begin{equation}
\label{e16d}
\int_0^T |(-A)^{\beta}\eta_\alpha ^{h}(t,x)|^2dt\le  T^{1/2-\beta} |h|^2, 
\end{equation}
 as claimed.

 \end{proof}

We are now ready to show
  \begin{Theorem}
  \label{t2}
   There exists  $C>0$ such that  for all       $\varphi\in L^2(H,\nu)$ and all $h\in D(A)$ we have
 \begin{equation}
\label{e14}
\int _H \langle D\varphi(x),h\rangle  \, \nu (dx) \le C\|\varphi\|_{L^2(H,\nu)} \,|Ah|_H.
\end{equation}

 \end{Theorem}
 \begin{proof} 
  Let us integrate   identity  \eqref{e7}  
   with respect to $\nu_\alpha$ over $H$. Taking into account the invariance of $\nu_\alpha$, we obtain
\begin{equation}
\label{e8}
\begin{array}{lll}
 \ds \int_H \langle D\varphi(x), h\rangle \nu_\alpha(dx)&=&\ds \int_H\langle DP^\alpha_t\varphi(x),h\rangle\,\nu_\alpha(dx)\\
 \\
&&- \ds  \int_0^t\int_H \langle Ah+p'_\alpha(x)h  ,DP^\alpha_s\varphi(x) \rangle ds \,\nu_\alpha(dx).
 \end{array}
\end{equation}
We are going to estimate
\begin{equation}
\label{e9}
\int_H\langle DP^\alpha_t\varphi(x),h\rangle\,\nu_\alpha(dx)
\end{equation}
  and
  \begin{equation}
\label{e10}
 \int_0^t\int_H \langle Ah+p'_\alpha(x) h  ,DP^\alpha_s\varphi(x) \rangle ds \,\nu_\alpha(x),
\end{equation}
  using  the Bismut--Elworthy-Li  formula \eqref{e4.10}. First notice that by \eqref{e4.10} it follows that
  $$
  \langle DP^\alpha_t\varphi(x), h\rangle^2\le\frac{1}{t^2}\;\E[\varphi^2(X_\alpha(t,x))]\;\E\int_0^t|(-A)^{\gamma/2}D_xX_\alpha(s,x)h|^2ds
  $$
  In view of Hypothesis \ref{h1} we can choose now $\beta<1/2$ such that
  $$
  \frac{n}2-1<\gamma<2\beta.
  $$
      The operator $ (-A)^{(\gamma-2\beta)/2}$ is bounded, set
  \begin{equation}
  \label{e21f}
 K:= \|(-A)^{(\gamma-2\beta)/2}\|.
  \end{equation}
 Then, since
$  (-A)^{\gamma/2}=(-A)^{(\gamma-2\beta)/2}(-A)^{\beta}
  $, 
  $$
  \E\int_0^t|(-A)^{\gamma/2}D_xX_\alpha(s,x)h|^2ds\le K^2\E\int_0^t|(-A)^{\beta}D_xX_\alpha(s,x)h|^2ds
  $$
  Taking into account \eqref{e13d}    we  find
 $$
 \langle DP^\alpha_t\varphi(x), h\rangle^2\le  \frac{K^2}{t^{1+2\beta}}\;P^\alpha_t(\varphi^2)(x)|h|^2.
$$
   Equivalently
   $$
  \langle DP^\alpha_t\varphi(x), h\rangle\le Kt^{-1/2-\beta}\;[P^\alpha_t(\varphi^2)(x)]^{1/2}\,|h|.
  $$
  Integrating with respect to  $\nu_\alpha$ over $H$, yields for a function $h\in L^2(H,\nu_\alpha)$,
  $$
  \begin{array}{l}
\ds  \int_H\langle DP^\alpha_t\varphi(x), h(x)\rangle\nu_\alpha(dx)\le Kt^{-1/2-\beta}\;\int_H[P^\alpha_t(\varphi^2)(x)]^{1/2}\,|h(x)|\nu_\alpha(dx)\\
\\
\ds\le K t^{-1/2-\beta}\;\left(\int_H P^\alpha_t(\varphi^2) \,\nu_\alpha(dx)\right)^{1/2}\;\left(\int_H |h(x)|^2\nu_\alpha(dx)\right)^{1/2},
  \end{array}
  $$
  that is, taking into account the invariance of  $\nu_\alpha$
  \begin{equation}
\label{e11}
\int_H\langle DP^\alpha_t\varphi(x), h(x)\rangle\nu_\alpha(dx)\le Kt^{-1/2-\beta}\;\|\varphi\|_{L^2(H,\nu_\alpha)} \;\|h\|_{L^2(H,{\nu_\alpha})}
\end{equation} 
  Now we can estimate \eqref{e9} and \eqref{e10}. As for \eqref{e9} we have by \eqref{e11}
  \begin{equation}
\label{e15b}
 \int_H\langle DP^\alpha_t\varphi(x), h\rangle\nu_\alpha(dx)\le   Kt^{-1/2-\beta}\;\|\varphi\|_{L^2(H,{\nu_\alpha})} \;|h| 
\end{equation}
and as for \eqref{e10} 
 \begin{equation}
\label{e16b}
 \begin{array}{l}
 \ds \int_0^t\int_H \langle Ah+{p'_\alpha}(x)h  ,DP^\alpha_s\varphi(x) \rangle ds \,d\nu_\alpha\\
 \\
 \ds\le {K\int_0^ts^{-1/2-\beta}\,ds}\,\|\varphi\|_{L^2(H,\nu)} \,(| Ah|+{\|p'_\alpha\,h\|_{L^2(H,\nu_\alpha})}
  \end{array}
\end{equation}
Now, using  \eqref{e11} and \eqref{e15b}  we deduce (recall that $1/2+\beta<1$)  
 \begin{equation}
\label{e24d}
\begin{array}{l}
 \ds \int_H \langle D\varphi(x), h\rangle \nu_\alpha(dx)\le   Kt^{-1/2-\beta}\;\|\varphi\|_{L^2(H,\nu_\alpha)} \;{|h|}  \\
 \\
\ds\ds{+\frac{2Kt^{1/2-\beta}}{1-2\beta}\;\|\varphi\|_{L^2(H,\nu)} \,(| Ah|+{\|p'_\alpha\,h\|_{L^2(H,\nu_\alpha})}}.
\end{array}
\end{equation}
 Setting  $t=1$ in \eqref{e24d} and letting $\alpha\to 0$, we arrive, recalling Proposition \ref{pA1} below, at 
{ \begin{equation}
\label{e14a}
\int _H \langle D\varphi(x),h\rangle  \, \nu (dx) \le C\|\varphi\|_{L^2(H,\nu)} (|Ah|+\|p'\,h\|_{L^2(H,\nu)}).
\end{equation}}
Thanks to Sobolev embedding, we choose $r> 2$ such that $D(A)\subset L^r(\mathcal O)$. Then thanks
to H\"older inequality and  \eqref{e7c}, there exists $C_r>0, C'_r>0$ such that
$$
\|p'\,h\|_{L^2(H,\nu)}\le C_r|h|_{L^{r}(\mathcal O)}\le   C'_r |Ah|,\quad\forall\;h\in D(A).
$$
   Thus the conclusion follows.

\end{proof}

\section{Some consequences of Theorem \ref{t2}}
\begin{Proposition}
\label{p2}
Assume that  estimate \eqref{e14} is fulfilled. Then     the linear operator
 $$
\varphi\in C^1_b(H)\mapsto (-A)^{-1} D\varphi\in C_b(H;H),$$
is closable in $L^2(H,\nu)$.\end{Proposition}
 
\begin{proof}   

 {\it Step 1}. For any  $h\in H$  the linear operator
 $$
\varphi\in C^1_b(H)\mapsto \langle (-A)^{-1} D\varphi(x),h\rangle \in C_b(H)$$
is closable in  $L^2(H,\nu).$\medskip

In fact, let  $(\varphi_n)\subset   C^1_b(H)$ and  $f\in L^2(H,\nu)$ be such that
 $$
\left\{\begin{array}{l}
\ds \varphi_n\to 0\;\;\mbox{\rm in}\;L^2(H,\nu),\\
\\\ds 
\langle (-A)^{-1} D\varphi_n(x),h\rangle \to f\;\;\mbox{\rm in}\;L^2(H,\nu).
\end{array}\right.
$$
We claim that  $f=0$.
 
Take  $\psi\in C^1_b(H)$,
then replacing in \eqref{e14} $\varphi$ by
 $\psi\varphi_n$, yields
 $$
\begin{array}{l}
\ds\left|\int_H   [\psi(x)((-A)^{-1}D\varphi_n(x) \cdot h)+\varphi_n(x)((-A)^{-1}D\psi(x)\cdot h)]\,\nu(dx)\right|\\
\\
\hspace{10mm}\ds \le   \|\varphi_n\psi\|_{L^2(H,\nu)}\,|h|_H
\le  \|\psi\|_\infty\;  \|\varphi_n\|_{L^2(H,\nu)}\,|h|_H.
\end{array}
$$
Letting  $n\to\infty$, we have
 $$
 \int_H   \psi(x) f(x)\,\nu(dx)=0,
$$
which yields  $f=0$ by the arbitrariness of  $\psi$, thereby proving the claim. \medskip

 {\it  Step 2}.  Conclusion.\medskip

    Let  $(\varphi_n)\subset   C^1_b(H)$ and  $F\in C_b(H;H)$ such that
 $$
\left\{\begin{array}{l}
\ds\varphi_n\to 0\;\;\mbox{\rm in}\;L^2(H,\nu),\\
\\
\ds(-A)^{-1} D\varphi_n \to F\;\;\mbox{\rm in}\;L^2(H,\nu;H).
\end{array}\right.
$$
We claim that $F=0$.\medskip

Let  $(e_k)$ be an orthonormal basis on $H$  consisting of eigenvectors of $A$ and let \ $(\alpha_k)$ such that
 $$
Ae_k=-\alpha_ke_k,\quad k\in\N.
$$
Then for any  $k\in\N$ we have  $$\langle (-A)^{-1} D_{k}\varphi_n(x),e_k\rangle\to \alpha_k^{-1}\langle F(x),e_k\rangle\quad\mbox{\rm  in }\;L^2(H,\nu).$$ 

By Step  1 taking  $h=e_k$ we see that  $D_k=D_{e_k}
$ is a closable operator on  $L^2(H,\nu)$ for any  $k\in \N$. 

 So, 
  $$\langle F(\cdot),e_k\rangle=0,\quad\forall\;k\in\N$$ which yields  $F=0$ as required.
 
 \end{proof}

 \subsection{The Sobolev space and the  integration by parts formula}
Let us   denote by   $W^{1,2}_A(H,\nu)$ the domain of the closure of   $(-A)^{-1} D$.   Denoting by $M^*$   the adjoint of   $(-A)^{-1} D$ we have
\begin{equation}
\label{e13}
 \int_H ((-A)^{-1} D\varphi(x)\cdot F(x))\,\nu(dx) =\int_H \varphi(x)\,M^*(F)(x)\,\nu(dx).
\end{equation}
Let  now  $h\in H$, set   $F^h(x)=h,\;\forall\,x\in H$.
By Theorem \ref{t2} we obtain
 $$
\int _H ((-A)^{-1}D\varphi(x)\cdot F^h(x)) \, \nu(dx) \le C\|\varphi\|_{L^2(H,\nu)} \,|Ah|_H, $$
so that  $F^h$ belongs  to the domain of  $M^*$. 
  
   Setting  $M^*(F^h)=v^h$,
we obtain the following integration by part formula.
 \begin{Proposition}
 \label{p4}
For any  $h\in H$ and  any  $\varphi\in W^{1,2}_A(H,\nu)$ there exists a function  $v^h\in L^2(H,\nu)$ such that
\begin{equation}
\label{e14b}
\int_H \langle (-A)^{-1} D\varphi(x),h\rangle\,\nu(dx)=\int_H \varphi(x)\,v^h(x)\,\nu(dx).
\end{equation}

\end{Proposition}

Therefore if   $h\in H$     there exists  the   Fomin derivative of   $\nu$ in the direction of   $A^{-1} h$. 
\begin{Remark}
\label{r5}
\em
Assume that $p=0$. Then $\mu=N_Q$,  where   $Q=-\frac12\;A^{-1}$.  Setting $v^h(x)=\sqrt 2\langle Q^{-1/2}x,h\rangle$  \eqref{e13}  reduces to the usual integration by parts formula for the Gaussian measure   $\mu$.  
 Notice, however,   that in this case we can take  $h\in D((-A)^{1/2})$
 
   \end{Remark}

   \appendix

   \section{Convergence of $\nu_\alpha$ to $\nu$}

     \begin{Proposition}
\label{pA1}
For all $\varphi\in C_b(H)$ we have
\begin{equation}
\label{eA1}
\lim_{\alpha\to 0}\int_H\varphi\,d\nu_\alpha=\int_H\varphi\,d\nu.
\end{equation}

\end{Proposition}
\begin{proof}
Let $\varphi\in C_b(H)$. Fix  $\alpha\in (0,1]$, $x\in H$ and write
\begin{equation}
\label{eA2}
\begin{array}{l}
\ds \left|\int_H\varphi\,d\nu- \int_H\varphi\,d\nu_\alpha   \right|\le \left|\int_H\varphi\,d\nu-P_t\varphi(x)   \right|\\
\\
\ds+\left| P_t\varphi(x)-P^\alpha_t\varphi(x)   \right|+\left|P^\alpha_t\varphi(x)- \int_H\varphi\,d\nu_\alpha   \right|.
\end{array} 
\end{equation}
Now choosing $\beta$  such that $\tfrac1n<\gamma<2\beta<1$, and taking into account \eqref{e13d}, we have
\begin{equation}
\label{e13dd}
\int_0^T |(-A)^{\gamma/2}\eta_\alpha ^{h}(t,x)|^2dt\le  KC_{T,\beta} T^{1-2\beta}|h|^2,\quad x\in H,
\end{equation}
where $K$ is defined in \eqref{e21f}.
It follows by Bismut-Elworthy formula that for a suitable constant $C>0$  we have
\begin{equation}
\label{eA3}
\begin{array}{l}
\ds \left|P^\alpha_t\varphi(x)- \int_H\varphi\,d\nu_\alpha   \right|=\left|\int_H[P^\alpha_t\varphi(x)- P^\alpha_t\varphi(y)]\,\nu_\alpha(dy)   \right|\\
\\
\ds\le C t^{-\beta}\|\varphi\|_\infty\int_H|x-y|\,\nu_\alpha(dy)
\le C t^{-\beta}\|\varphi\|_\infty\left((|x|+\int_H|y|\,\nu_\alpha(dy)\right).
\end{array} 
\end{equation}\medskip

{\it Claim}. There exists $M>0$ such that
\begin{equation}
\label{eA4}
\int_H|y|\,\nu_\alpha(dy)\le M,\quad \forall\;\alpha\in(0,1].
\end{equation}
Once the claim is proved the conclusion follows easily from \eqref{eA2} and \eqref{eA3} and \cite[Theorem 4.16]{Da04}. \footnote{Since $p_\alpha$ is dissipative the proof that
$\lim_{\alpha\to 0}P^\alpha_t\varphi(x)=\int_H\varphi\,d\nu_\alpha$ is exactly the same as that in  \cite[Theorem 4.16]{Da04}. }\medskip

To prove the claim it is enough to show 
\begin{equation}
\label{eA5}
\E|X_\alpha(t,x)| \le M,\quad \forall\;\alpha\in(0,1].
\end{equation}
This can be proved as the estimate (4.13) in \cite{Da04} taking into account that
for any $m\in\N$ there is $K_m>0$ such that
\begin{equation}
\label{e6.36}
\E\int_\mathcal O|W_A(t,\xi)|^{2m}d\xi\le K_m,\quad\forall\;t\ge 0.
\end{equation}
Here is finally the proof of \eqref{e6.36}.
We start from the identity
$$
W_A(t,\xi)= \sum_{k\in\N^n}\int_{0}^{t}e^{-\pi^2|k|^2(t-s)} (\pi^2|k|^2)^{-\gamma/2}e_k(\xi)dW_k(s),$$
where 
$$e_k(\xi)=(2\pi)^{n/2}\sin(k_1\xi_1)\cdots\sin(k_n\xi_n),\quad k=(k_1,...,k_n).
$$
Therefore for each $(t,\xi)\in [0,+\infty)\times\mathcal O$, $W_A(t,\xi)$ is a real Gaussian variable with mean $0$ and covariance $q(t,\xi)$ given by
\begin{equation}
\label{e6.37}
q(t,\xi)=\sum_{k\in\N^n}\int_{0}^{t}e^{-2\pi^2|k|^2(t-s)} (\pi^2|k|^2)^{-\gamma}|e_k(\xi)|^2ds.
\end{equation}
Since $|e_k(\xi)|^2\le (2/\pi)^n$ and $\gamma>\frac{n}2 -1$, we find
\begin{equation}
\label{e6.38}
q(t,\xi)\le C(n,\gamma)\sum_{k\in\N^n}\frac1{|k|^{2+2\gamma}}=C_1(n,\gamma)<\infty,\quad\forall\;t\ge 0.
\end{equation}
Therefore
\begin{equation}
\label{e6.39}
\E (W_A(t,\xi))^{2m} \le C_2(n,\gamma,m),\quad\forall\;t\ge 0.
\end{equation}
Finally, integrating in $\xi$ over $\mathcal O$, yields
$$
\E\int_\mathcal O(W_A(t,\xi))^{2m}d\xi\le C_2(n,\gamma,m)\;\mbox{\rm meas.}\;(\mathcal O),\quad\forall\;t\ge 0,
$$
and  the conclusion follows. 
\begin{Remark}
We have used $p'\le 0$ in \eqref{eA3}. If we assume only that $p'$ is bounded above, the differential 
of the transition semigroup may grow in time. However, using classical arguments (see for instance
\cite{De13}), convergence to the invariant measure can be proved under this more general assumption.
\end{Remark}

\end{proof}

\end{document}